\documentclass[11pt]{amsart}

\usepackage{amsmath, amssymb,amsthm}
\usepackage{MnSymbol}
\usepackage{color}
\usepackage{graphicx}
\usepackage{colortbl}
\usepackage{mathtools}
\usepackage{float}
\usepackage[backend=bibtex]{biblatex}
\newtheorem{theorem}{Theorem}[section]

\newtheorem{lemma}[theorem]{Lemma}

\newtheorem{example}{Example}
\newtheorem{proposition}[theorem]{Proposition}
\newtheorem{corollary}[theorem]{Corollary}

\newtheorem{conjecture}{Conjecture}

\newlength\cellsize 

\savebox2{%
\begin{picture}(15,15)
\put(0,0){\line(1,0){15}}
\put(0,0){\line(0,1){15}}
\put(15,0){\line(0,1){15}}
\put(0,15){\line(1,0){15}}
\end{picture}}

\savebox3{%
\begin{picture}(15,15)
\put(0,0){\line(1,0){17}}
\put(0,15){\line(1,0){17}}
\end{picture}}

\newcommand\cellify[1]{\def\thearg{#1}\def\nothing{}%
\ifx\thearg\nothing\vrule width0pt height\cellsize depth0pt%
  \else\hbox to 0pt{\usebox2\hss}\fi%
  \vbox to 15\unitlength{\vss\hbox to 15\unitlength{\hss$#1$\hss}\vss}}

\newcommand\tableau[1]{\vtop{\let\\=\cr
\setlength\baselineskip{-12000pt}
\setlength\lineskiplimit{12000pt}
\setlength\lineskip{0pt}
\halign{&\cellify{##}\cr#1\crcr}}}

\newcommand{\graybox}{\textcolor[RGB]{165,165,165}{\rule{1\cellsize}{1\cellsize}}\hspace{-\cellsize}\usebox2}

\newcommand{\grayboxtwo}{\textcolor[RGB]{220,220,220}{\rule{1\cellsize}{1\cellsize}}\hspace{-\cellsize}\usebox2}

\title[Some conjectures on the Schur expansion of Jack Polynomials]{Some conjectures on the Schur Expansion \\ of Jack Polynomials}  

\author[Per Alexandersson]{Per Alexandersson}
\address{Department of Mathematics, KTH Royal Institute of Technology, Stockholm, Sweden}
\email{per.w.alexandersson@gmail.com}

\author[James Haglund]{James Haglund}
\address{Department of Mathematics, University of Pennsylvania, Philadelphia, PA, 19104}
\email{jhaglund@math.upenn.edu}

\author[George Wang]{George Wang}
\address{Department of Mathematics, University of Pennsylvania, Philadelphia, PA, 19104}
\email{wage@math.upenn.edu}


\date{\today}


\begin{document}

\maketitle

\begin{abstract}
    We present positivity conjectures for the Schur expansion of Jack symmetric functions in two bases given by binomial coefficients. Partial results suggest that there are rich combinatorics to be found in these bases, including Eulerian numbers, Stirling numbers, quasi-Yamanouchi tableaux, and rook boards. These results also lead to further conjectures about the fundamental quasisymmetric expansions of these bases, which we prove for special cases.
\end{abstract}

\section{Introduction}
The (integral form, type $A$) Jack polynomials $J_\mu ^{(\alpha)} (X)$ are an important family of 
symmetric functions with applications to many areas, including statistics, mathematical physics, 
representation theory, and algebraic combinatorics. They depend on a set of variables $X$ and a 
parameter $\alpha$, and specialize into several other families of symmetric polynomials: monomial 
symmetric functions $m_\mu$ $(\alpha = \infty)$,  elementary symmetric functions 
$e_{\mu'}$ $(\alpha = 0)$, Schur functions $s_\mu$ $(\alpha = 1)$, and zonal 
polynomials $(\alpha = 1/2,\ \alpha = 2)$, each significant in their own right.  

Despite their relations to many well studied families of polynomials, Jack polynomials are comparatively poorly understood. 
One area that has seen some progress is their positivity in other bases.  From the definition of the $J_{\mu}^{(\alpha)}$, it is not obvious that the coefficients of the monomial expansion are in $\mathbb{Z}[\alpha]$, but this integrality conjecture was proven by Lapointe and Vinet  \cite{LaVi95b}.   A celebrated result of Knop and Sahi \cite{KnSa97} obtained later gives an explicit combinatorial formula for the expansion of $J_{\mu}^{(\alpha)}$ in the monomial basis, implying the stronger result that the coefficients lie in $\mathbb{N}[\alpha]$.  

Up until now there have not been any conjectures involving the expansion of $J_{\mu}^{(\alpha)}$  in the Schur basis;  the integrality result of 
Lapointe and Vinet implies these coefficients are in $\mathbb Z [\alpha]$, but computations show that they are
not generally in $\mathbb N [\alpha]$.    However, we have discovered that if we first define 
$$
{\tilde J}_{\mu}^{(\alpha)} (X) = \alpha ^n J_{\mu}^{(1/\alpha)} (X),
$$
then take the the coefficient of a given Schur function $s_{\lambda}(X)$ in ${\tilde J}_{\mu}^{(\alpha)} (X)$ 
and expand it either in the basis $\{ {\alpha + k \choose n } \}$ or in $\{{\alpha\choose k} k!\}$, the coefficients
appear to be nonnegative integers.  This opens up the intriguing question of whether or not these nonnegative integers have 
a combinatorial interpretation.   So far we been unable to find such an interpretation for general $\mu, \lambda$, but we hope this paper will inspire further research in this direction which will eventually lead to a solution to this question.

Our work on the $J_{\mu}^{(\alpha)}$ grew out of a conjecture about the (integral form, type A) Macdonald polynomials $J_{\mu}(X;q,t)$.
For $\mu$ a partition of $n$, it is well-known that
\begin{align*}
J_{\mu}^{(\alpha)}(X) = \lim _{t \to 1} J_{\mu}(X;t^{\alpha},t)/(1-t)^n,
\end{align*}
(see \cite[Chapter 6]{Macdonald} for background on Jack polynomials and Macdonald polynomials).  
The second author conjectures that for $k \in \mathbb N$ and $\mu, \lambda$ partitions of $n$,
\begin{align}
\label{Mconjecture}
\langle J_{\mu}(X;q,q^k)/(1-q)^n, s_{\lambda}(X) \rangle \in \mathbb N[q],
\end{align}
where $\langle\ ,\ \rangle$ is the usual Hall scalar product with respect to which the Schur functions are orthonormal, so Schur positivity seems to hold in a certain sense.   (It is a famous result of Mark Haiman \cite{Hai01} that the coefficients obtained by 
expanding $J_{\mu}(X;q,t)$ into the ``plethystic Schur" basis
$\{s_{\lambda}[X(1-t)]\}$ are in $\mathbb N[q,t]$, but little is known about expansions of $J_{\mu}(X;q,t)$ into the Schur basis.)  
The combinatorial formula of
Haglund, Haiman, and Loehr for the $J_{\mu}(X;q,t)$ \cite{HHL05b} implies that the coefficients obtained when expanding $J_{\mu}(X;q,q^k)/(1-q)^n$ into 
monomial symmetric functions are in $\mathbb N[q]$, and Yoo \cite{Yoo12, Yoo15} has proven that (\ref{Mconjecture}) holds
for certain important special cases, including when $\mu$ has two columns and when $\mu = (n)$.  

Beyond proving (\ref{Mconjecture}) 
in general, one could hope to be able to replace the nonnegative integer $k$ in the conjecture by a 
continuous parameter $\alpha$.  It is fairly easy to show that we have
\begin{align*}
{\tilde J}_{\mu}^{(\alpha)}(X) = \lim _{q \to 1} J_{\mu}(X;q,q^{\alpha})/(1-q)^n,
\end{align*}
which suggests that investigating the Schur expansion of $\tilde{J}_\mu^{(\alpha)}(X)$ may shed some light on this case. In this article we introduce the 
following conjectures for this expansion, which have been tested using J. Stembridge's Maple package SF 
\cite{STE}.

\begin{conjecture}
\label{C1}
Let $\mu$, $\lambda$ be partitions of $n$. Then setting
$$\langle \tilde{J}_\mu^{(\alpha)}(X), s_\lambda \rangle = \sum_{k=0}^{n-1}  a_{k}(\mu,\lambda){\alpha + k \choose n},$$
we have $a_{k}(\mu,\lambda) \in \mathbb{N}$.  Furthermore, the polynomial
$\sum_{k=0}^n a_k(\mu,\lambda) z^k$ has only real zeros.
\end{conjecture}

\begin{conjecture}
\label{C2}
Let $\mu$, $\lambda$ be partitions of $n$. Then setting
$$\langle \tilde{J}_\mu^{(\alpha)}(X), s_\lambda \rangle = \sum_{k=1}^n b_{n-k}(\mu,\lambda) {\alpha \choose k} k! ,$$
we have $b_{n-k}(\mu,\lambda) \in \mathbb{N}$. Furthermore, the polynomial $\sum_{k=0}^n b_{n-k}(\mu,\lambda)z^k$ has only real zeros.
\end{conjecture}

It turns out that part one of Conjecture \ref{C1} (almost) implies Conjecture \ref{C2}.  The identity
$
{\alpha + k  \choose n} = \sum _{i} {\alpha \choose i} {k \choose n-i}
$
shows that if the $a_{k}(\mu,\lambda) \in \mathbb N$, then $k! b_{n-k}(\mu,\lambda) \in \mathbb N$, so if Conjecture \ref{C1} is true, the only issue is
whether or not the $b_{n-k}(\mu,\lambda)$ are integers.

Using Yoo's results on (\ref{Mconjecture}) and other methods, we prove various special cases of Conjectures $1$ and $2$, which suggest that there are rich combinatorics lurking in these expansions. In particular, we find that both Eulerian and Stirling numbers appear, as well as rook boards and generating functions of quasi-Yamanouchi tableaux. Our results also point towards attractive conjectures for the fundamental quasisymmetric expansion in each basis.

\section{Preliminaries}

A \textit{partition} $\mu = (\mu_1 \geq \mu_2 \geq \cdots \geq \mu_k > 0)$ is a finite sequence of non-increasing positive integers. The \textit{size} of $\mu$, denoted $|\mu|$, is the sum of the integers in the sequence. The \textit{length} of $\mu$ is the number of integers in the partition, denoted $\ell(\mu)$. We say that $\mu$ \textit{dominates} $\lambda$ if $\mu_1 + \cdots + \mu_i \geq \lambda_1 + \cdots + \lambda_i$ for all $i \geq 1$. If $\mu$ dominates $\lambda$, then we write $\mu \geq \lambda$, forming a partial order on partitions.
We use French notation to draw the \textit{diagram} corresponding to a partition $\mu$ by having left justified rows of boxes starting at the bottom, where the $i$th row has $\mu_i$ boxes. The conjugate of $\mu$ is obtained by taking the diagram of $\mu$ and reflecting across the diagonal.

Bases of the ring of symmetric functions are indexed by partitions, and in particular we write $e_\mu$, $h_\mu$, $m_\mu$, and $p_\mu$ to denote the elementary, complete homogeneous, monomial, and power-sum symmetric functions respectively and $s_\mu$ to denote the Schur functions. We also write $F_\sigma(x)$ for the fundamental quasisymmetric function, where $\sigma \subseteq \{1,\ldots, n-1\}$ and 
$$F_\sigma(x) = \sum_{\substack{i_1 \leq \cdots \leq i_n \\ j\in \sigma \Rightarrow i_j < i_{j+1}}} x_{i_1}\cdots x_{i_n}.$$





\subsection{Eulerian and Stirling numbers}

A permutation $\pi = \pi_1 \pi_2 \cdots \pi_n \in S_n$ has a descent at position $i$ if $\pi_i > \pi_{i+1}$. Write $A(n,k)$ for the Eulerian number counting permutations in $S_n$ with $k$ descents and $S(n,k)$ for the Stirling number of the second kind counting the ways to partition $n$ labelled objects into $k$ nonempty, unlabelled subsets. 

For $|\lambda| = n$, we define the set of \emph{$\lambda$-restricted permutations} to be permutations where $1, 2,\ldots, \lambda_1$ appear in order, $\lambda_1+1, \ldots, \lambda_1 +\lambda_2$ appear in order, and so on. We define the $\lambda$-restricted Eulerian number $A(\lambda,k)$ to be the the number of $\lambda$-restricted permutations in $S_n$ with $k$ descents.
\begin{example}
The $10$ $(3,2)$-restricted permutations are $12345$, $12435$, $14235$, $41235$, $12453$, $14253$, $41253$, $14523$, $41523$, and $45123$.
\end{example}

\subsection{Quasi-Yamanouchi Tableaux}
A \textit{semistandard Young tableau} $T$ is a filling of the diagram of a partition $\mu$ using positive integers that weakly increase to the right and strictly increase upwards. We say that $T$ has shape $\mu$ and write SSYT$_m(\mu)$ to denote the set of semistandard Young tableaux of shape $\mu$ with maximum value at most $m$.

If $w_i = w_i(T)$ is the number of entries of $T$ with value $i$, then we say that $T$ has \textit{weight} $w = (w_1, w_2, \ldots)$. The Kostka number $K_{\mu\lambda}$ counts the number of semistandard Young tableaux of shape $\mu$ and weight $\lambda$. 
Let $|\mu| = n$. A \textit{standard Young tableau} is a semistandard Young tableau with weight $(1^n)$. The set of standard Young tableaux of shape $\mu$ is SYT$(\mu)$.

We say that an entry $i$ is weakly left of an entry $j$ in a tableau $T$ when $i$ is above or above and to the left of $j$. The \textit{descent set} of a tableaux $T$ is the set of entries $i \in Des(T) \subseteq \{1,\ldots,n-1\}$ such that $i+1$ is weakly left of $i$, and we write $des(T) = |Des(T)|$. Given a descent set $Des(T) = \{d_1, d_2,\ldots,d_{k-1}\}$, the $i$th run of $T$ is the set of entries from $d_{i-1}+1$ to $d_{i}$, where $1 \leq i \leq k$, $d_0 = 0$, and $d_k = n$.

\begin{example} A standard Young tableau with five runs. The first and fourth runs are highlighted, and the tableau has descent set $\{3,6,8,11\}$.
\begin{displaymath}
	\tableau{
	\mathclap{\graybox}\mathclap{\raisebox{4\unitlength}{9}}&\mathclap{\graybox}\mathclap{\raisebox{4\unitlength}{10}}& 12\\
	4&5&7&\mathclap{\graybox}\mathclap{\raisebox{4\unitlength}{11}}\\
	\mathclap{\grayboxtwo}\mathclap{\raisebox{4\unitlength}{1}}&\mathclap{\grayboxtwo}\mathclap{\raisebox{4\unitlength}{2}}&\mathclap{\grayboxtwo}\mathclap{\raisebox{4\unitlength}{3}} & 6&8\\
	}
	\end{displaymath}
\end{example}

A semistandard Young tableau is \textit{quasi-Yamanouchi} if when $i$ appears in the tableau, the leftmost instance of $i$ is weakly left of some $i-1$. QYT$_{\leq m}(\mu)$ is the set of quasi-Yamanouchi tableaux with maximum value at most $m$, and QYT$_{=m}(\mu)$ is the set of quasi-Yamanouchi tableaux with maximum value exactly $m$. 

\begin{example} All the quasi-Yamanouchi tableaux of shape $(2,2,1)$, demonstrating that $\textnormal{QYT}_{=3}(2,2,1)=3$ and $\textnormal{QYT}_{=4}(2,2,1) = 2$.
\begin{displaymath}
	\tableau{
	3\\
	2&2\\
	1&1\\}
	\ \ \ \ \ \tableau{
	3\\
	2&3\\
	1&1\\}
	\ \ \ \ \ \tableau{
	3\\
	2&3\\
	1&2\\}
	\ \ \ \ \ \tableau{
	4\\
	2&3\\
	1&2\\}
	\ \ \ \ \ \tableau{
	3\\
	2&4\\
	1&3\\}
	\end{displaymath}
\end{example}

Quasi-Yamanouchi tableaux first arose as objects of interest in the work of Assaf and Searles \cite{AsSe17}, where they showed that quasi-Yamanouchi tableaux could be used to tighten Gessel's expansion of Schur polynomials into fundamental quasisymmetric polynomials. Some of their combinatorial properties and a partial enumeration can be found in \cite{Wa16}.

A property of particular interest is that quasi-Yamanouchi tableaux of shape $\mu$ have a natural bijection with standard Young tableaux of shape $\mu$ via the destandardization map that Assaf and Searles use \cite[Definition 2.5]{AsSe17}.
Given a semistandard Young tableau $T$, choose $i$ such that the leftmost $i$ is strictly right of the rightmost $i-1$, meaning to the right or below and to the right, or such that there are no $i-1$, and decrement every $i$ to $i-1$. Repeat until no more entries $i$ can be decremented. The resulting tableau is the \textit{destandardization} of $T$. For standard Young tableaux, this is a bijection, and the inverse gives the \emph{standardization} of a quasi-Yamanouchi tableau.

\begin{proposition}[\cite{AsSe17,Wa16}]
Let $\mu = (\mu_1, \mu_2, \ldots, \mu_k)$ be a partition of size $n$, then
$\textnormal{QYT}_{\leq n} (\mu) \cong \textnormal{SYT}(\mu).$
\end{proposition}

Note that by definition, QYT$_{=m}(\mu)=0$ for any $|\mu| = n$ and $m>n$, so QYT$_{\leq n}(\mu)$ contains all quasi-Yamanouchi tableaux of shape $\mu$. Since we can partition QYT$_{\leq n}(\mu)$ into $\{\textrm{QYT}_{=m}(\mu)\  |\ 1\leq m \leq n\}$, this bijection gives a refinement on standard Young tableaux. In particular, the image of QYT$_{=m}(\mu)$ is the subset of SYT$(\mu)$ with exactly $m$ runs, or equivalently, $m-1$ descents.
We will also use the following result, which is obtained through a bijection consisting of standardizing a quasi-Yamanouchi tableau, conjugating, and then destandardizing.
\begin{lemma}[\cite{Wa16}]\label{QYTsymmetry}
Given a partition $\lambda$ of $n$, its conjugate $\lambda'$, and $1\leq k \leq n$, $\textnormal{QYT}_{=k}(\lambda) \cong \textnormal{QYT}_{=(n+1)-k}(\lambda')$.
\end{lemma}

\subsection{Dual Equivalence}
We will need Assaf's dual equivalence graphs \cite{As15a}, although not in their full generality.
Define the \textit{elementary dual equivalence involution} $d_i$ on $\pi \in S_n$ for $1<i<n$ by $d_i(\pi) = \pi$ if $i-1$, $i$, and $i+1$ appear in order in $\pi$ and by $d_i(\pi) = \pi'$ where $\pi'$ is $\pi$ with the positions of $i$ and whichever of $i\pm1$ is further from $i$ interchanged when they do not appear in order. Two permutations $\pi$ and $\tau$ are \textit{dual equivalent} when $d_{i_1}\cdots d_{i_k}(\pi) = \tau$ for some $i_1,\ldots,i_k$. The reading word of a tableau is obtained by reading the entries from left to right, top to bottom, which for standard Young tableaux produces a permutation, and two standard Young tableaux of the same shape are dual equivalent if their reading words are.

We also use Assaf's characterization of Gessel's expansion of the Schur function into the fundamental quasisymmetric basis,
$$s_\mu = \sum_{T\in [T_\mu]} F_{Des(T)}(x),$$
where $[T_\mu]$ is the dual equivalence class of all standard Young tableaux of shape $\mu$, which is in fact all standard Young tableaux of shape $\mu$.

\subsection{Robinson-Schensted Correspondence} The Robinson-Schensted correspondence is a bijective algorithm between permutations in $S_n$ and pairs of standard Young tableaux $(P,Q)$ where $P$ and $Q$ have the same shape. For $\pi \in S_n$ written in one line notation, we apply an insertion procedure as follows.

Given a semistandard Young tableau $T$, insert the value $x_1$ by scanning for the first entry in the first row from the left which is larger than $x_1$. If none exists, then adjoin a new cell with $x_1$ to the end of this row and terminate the procedure. If such an $x_2>x_1$ does exist, replace its entry with $x_1$ and scan the second row for the first entry from the left larger than $x_2$.  If none exists, adjoin $x_2$ to the end of the second row. If such an $x_3>x_2$ exists, replace its entry with $x_2$ and repeat this process in the third row. Eventually this procedure must terminate, and we are left with a new tableau $T'$.

The Robinson-Schensted correspondence takes $\pi = \pi_1 \cdots \pi_n$ in one line notation and successively inserts $\pi_1, \ldots, \pi_n$ into the empty diagram to get the insertion tableau $P$. $Q$ is obtained by recording the order in which cells of $P$ are added by adjoining a cell containing $i$ after the insertion of $\pi_i$ such that the insertion tableau and recording tableau maintain the same shape at every step. We will refer to this algorithm as RSK, after the more general Robinson-Schensted-Knuth correspondence.

\subsection{Rook Boards}

Given an $n\times n$ grid, we can choose a subset $B$, which we call a \emph{board}. The $k$th \emph{rook number} of $B$, denoted $r_k(B)$, is the number of ways to place $k$ nonattacking rooks on $B$, and the $k$th \emph{hit number}
of $B$, denoted $h_k(B)$, is the number of ways to place $n$ nonattacking rooks on the grid with exactly $k$ on $B$. A \emph{Ferrers board} is one where if $(x,y)$ is in $B$, then every $(i,j)$ weakly southeast is also in $B$.
We will use the following result from Goldman, Joichi, and White \cite{GJW75} which translates certain products of factors into each of our bases.

\begin{proposition}\label{rookboardprop}
Let $0\leq c_1\leq c_2\leq \cdots \leq c_n \leq n$ with $c_i \in \mathbb{N}$, and let $B = B(c_1,\ldots, c_n)$ be the Ferrers board whose $i$th column has height $c_i$. Then
$$\prod_{i=1}^n (\alpha + c_i -i + 1) = \sum_{k=0}^n h_k(B){\alpha +k\choose n} = \sum_{k=0}^n r_{n-k}(B) {\alpha \choose k} k!$$
\end{proposition}

\section{Partial Results} 
\subsection{Eulerian and Stirling numbers}
We first became interested in Conjecture 1 and Conjecture 2 due to the following observation, which follows from the normalization property of Jacks and the way that $\alpha^n$ is written in these two bases.
\begin{proposition} \label{eulerianstirlingprop}
For a partition $\mu$, the coefficient of $m_{1^n}$ in $\tilde{J}_\mu^{(\alpha)}(X)$ is
$$\langle \tilde{J}_\mu^{(\alpha)}(X), h_{1^n}\rangle = n!\alpha^n = \sum_{k=0}^{n-1} n! A(n,k) {\alpha + k \choose n} = \sum_{k=1}^{n} n! S(n,k){\alpha \choose k} k!.$$
\end{proposition}
Combined with computer data confirming the two conjectures up to $n=11$, we have the first hint that these bases may have interesting combinatorial interpretations. Two immediate corollaries come from extracting the coefficient of $m_{1^n}$ from $s_\lambda$ in the Schur expansion of $\tilde{J}_\mu^{(\alpha)}(X)$.
\begin{corollary} \label{eulerianstirlingcor1}
Given a partition $\mu$, we have
$$\sum_{|\lambda| = n} \sum_{k=0}^{n-1} a_{k}(\mu,\lambda) K_{\lambda(1^n)} {\alpha + k \choose n} = \sum_{k=0}^{n-1} n!A(n,k){\alpha + k \choose n}.$$
\end{corollary}
\begin{corollary}
Given a partition $\mu$, we have
$$\sum_{|\lambda| = n} \sum_{k=1}^{n-1} b_{n-k}(\mu,\lambda) K_{\lambda(1^n)} {\alpha \choose k} k! = \sum_{k=1}^{n-1} n!S(n,k){\alpha \choose k} k!.$$
\end{corollary}
Furthermore, if $a_k(\mu,\lambda) \in \mathbb{N}[\alpha]$ or $b_k(\mu,\lambda) \in \mathbb{N}[\alpha]$ in general, then the respective result above would indicate some refinement on Eulerian numbers or Stirling numbers of the second kind.

\subsection{Quasi-Yamanouchi Tableaux}

In the case of $\mu = (n)$ and $|\lambda| = n$, we noticed that the equality
$$\sum_{k=0}^{n-1} a_k((n),\lambda) = |\textrm{SYT}(\lambda)|$$
held for the computer generated data. Upon closer inspection, it appeared that in fact the following theorem was true.

\begin{theorem}\label{QYTtheorem}
Let $\lambda$ be a partition of $n$ and $\lambda'$ be its conjugate. Then for the coefficient of $s_\lambda$ in $\tilde{J}_\mu^{(\alpha)}(X)$
$$\langle \tilde{J}_{(n)}^{(\alpha)}(X), s_\lambda\rangle = \sum_{k=0}^{n-1} a_k((n),\lambda) {\alpha + k \choose n}$$
we have $a_k((n),\lambda) =n!  \textnormal{QYT}_{=k+1}(\lambda')$.
\end{theorem}

We split the proof into several parts, starting with the coefficient of $m_\lambda$ in $J_\mu^{(\alpha)}(X)$. By example 3 in chapter IV, section 10 of Macdonald \cite{Macdonald}, this is $\frac{n!}{\lambda!}\prod_{s\in \lambda} (arm(s)\alpha +1)$ where $\lambda! = \lambda_1!\lambda_2!\cdots$ and the arm of a cell $s$ in the diagram of $\lambda$ is the number of cells to the right of $s$. Converting to $\tilde{J}_\mu^{(\alpha)}(X)$, this becomes $\frac{n!}{\lambda!}\prod_{s\in \lambda} (\alpha + arm(s))$. The next step is to convert these coefficients to the new basis.

\begin{lemma}
Given a partition $\lambda$ of $n$,
$$\frac{n!}{\lambda!}\prod_{s\in \lambda} (\alpha + arm(s)) = n! \sum_{k=0}^{n-1} A(\lambda, k){\alpha + n - 1 - k \choose n}$$
where $A(\lambda,k)$ is the number of $\lambda$-restricted permutations with $k$ descents. 
\end{lemma}

\begin{proof}
Cancel $n!$ and rewrite the left hand side to get 
$$\prod_{i=1}^{\ell(\lambda)} {\alpha + \lambda_i -1\choose \lambda_i} =\sum_{k=0}^{n-1} A(\lambda, k){\alpha + n - 1 - k \choose n}.$$ 
We show that this equality holds with a bijection. Assume $\alpha \in \mathbb{N}$ and $\alpha \geq n$. On the left hand side we count diagrams where we take a rectangle of cells with $\alpha-1$ rows and $\ell(\lambda)$ columns, then adjoin the conjugate shape of $\lambda$ at the bottom. In the $i$th column of the diagram, we choose $\lambda_i$ many cells and mark them with dots. On the right hand side, we count pairs where, for some $k$, the first element is a $\lambda$-restricted permutation with $k$ descents and the second element is a column of cells of height $\alpha + n -1 -k$ with $n$ cells marked by dots.
Given a diagram counted by the left hand side, we apply the following algorithm to get a pair counted by the right hand side.

\begin{enumerate}
\item  Label the dots in the diagram so that the first column's dots read $1,\ldots, \lambda_1$ from top to bottom, the second column's dots read $\lambda_1 +1,\ldots,\lambda_2$ from top to bottom, etc. Extend the diagram downwards by adjoining cells to the bottom (without moving any dots) so that it becomes a rectangle of height $\alpha + n-1$ and width $\ell (\lambda)$. Set $i = 1$ and start with a pair where the first element is the empty word and the second element is a column of height $0$. 
\item Read across row $i$, where rows are counted starting from the top. If there is no dot, go to 3a. If there is a dot in this row, go to 3b. 
\item \begin{enumerate}
    \item Do nothing to the word in your pair and adjoin a blank cell to the bottom of your column. Go to step 4. 
    \item Adjoin the label of the first dot from the left in this row to the end of the word in your pair. If this is not a descent, add a new cell to the bottom of the column in your pair and mark it with a dot. If this is a descent, then do not add a new box to the bottom of the column, but do mark the bottom cell in the column (which will be blank if it is a descent) with a dot. Delete the dot that was hit and and push all dots that are not in the same column down by one row. Go to step 4.
    \end{enumerate}
\item If $i = \alpha + n-1$, then terminate the algorithm, else increment $i$ by one and go back to step $2$. 
\end{enumerate}
\begin{example} An example of the algorithm for $\lambda=(2,2,1)$.
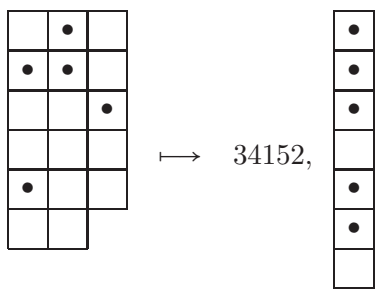
\begin{center}
\begin{displaymath}
$$\tableau{
\ & \bullet & \ \\
\bullet & \bullet & \ \\
\ & \ & \bullet \\
\ & \ & \ \\
\bullet & \ & \ \\
\ & \ \\
}\ \ \ \raisebox{-43pt}{$\longmapsto$} \ \  \raisebox{-43pt}{$34152,$}\ \ \tableau{\bullet \\ \bullet \\ \bullet \\ \ \\ \bullet \\ \bullet \\ \ \\}$$ 
\end{displaymath}
\end{center}
\end{example}

The algorithm must terminate, because the loop always goes through step 4. To show that it is well-defined, we need to check that no dot can be pushed below row $\alpha +n -1$. If we consider a dot in column $i$ at the lowest possible position, row $\alpha+\lambda_i -1$, then it needs to be pushed down $n-\lambda_i +1$ times to leave the diagram. However, there are only $n-\lambda_i$ dots outside of this column that can contribute to pushing this dot down, so no dot can be pushed outside of the diagram. 

We can obtain an inverse by reversing the steps of the algorithm. In this direction, the marked column encodes the row that a dot comes from and the $\lambda$-restricted permutation encodes which column a dot comes from. By similar reasoning as above, we also have to end with all dots in column $i$ at or above row $\alpha+\lambda_i-1$. 

This algorithm gives a bijection that holds for any $\alpha \geq n$. Since both sides of the equality we are trying to prove are finite degree polynomials in $\alpha$, this is sufficient to prove equality.
\end{proof}

We now wish to relate these $\lambda$-restricted Eulerian numbers to quasi-Yamanouchi tableaux. We can achieve this using RSK. 
\begin{lemma} \label{RSKlemma}
Given a partition $\lambda$ of $n$, it holds that
$$\sum_{k=0}^{n-1} A(\lambda,k){\alpha +n -1-k \choose n} = \sum_{\substack{|\nu| = n \\ \nu \geq \lambda}} K_{\nu\lambda} \sum_{k=0}^{n-1} \textnormal{QYT}_{=k+1}(\nu){\alpha + n -1-k \choose n}.$$
\end{lemma}
\begin{proof}
By comparing coefficients of ${\alpha + n -1 -k\choose n}$, it is sufficient to show that for a fixed $k$,
$$A(\lambda, k) = \sum_{\substack{|\nu|=n \\ \nu \geq \lambda}}K_{\nu\lambda}\textnormal{QYT}_{=k+1}(\nu).$$
We prove this through a bijection between $\lambda$-restricted permutations with $k$ descents and pairs of tableaux $(P,Q)$ of the same shape, where $P$ is a standard Young tableau with $k+1$ runs and $Q$ is a semistandard Young tableaux with weight $\lambda$. 

Given a $\lambda$-restricted permutation $\sigma$ with $k$ descents, obtain $\sigma'$ by decrementing all integers $1+\sum_{i=0}^j \lambda_i,\ldots, \sum_{i=0}^{j+1} \lambda_i$ to $j+1$ where $\lambda_0 = 0$ for all $0\leq j < \ell(\lambda)$. Create a two line array with $\sigma'$ in the top row and the integers $1,\ldots,n$ in order on the bottom, then reorder this array so that the columns have pairs in lexicographic order. Map this array via RSK to a pair $(P,Q)$, where $P$ is a standard Young tableau and $Q$ is a semistandard Young tableau with weight $\lambda$. We want to show that $P$ has $k+1$ runs. 

If $\sigma_i < \sigma_{i+1}$, then $i+1$ is inserted after $i$ in $P$, and RSK will keep $i$ and $i+1$ in the same run of $P$. If $\sigma_i > \sigma_{i+1}$, then $i+1$ must be inserted before $i$. In this case, RSK will force $i+1$ to stay weakly left of $i$. Thus, descents in $\sigma$ correspond to descents in $P$, and $P$ has $des(P)+1 = k+1$ runs. This shows that RSK maps the two line arrays defined by $\lambda$-restricted permutations to the desired set of pairs $(P,Q)$. It remains to show that the inverse map has image contained in the $\lambda$-restricted permutations.

Take some pair of tableaux $(P,Q)$ where $P$ is standard with $k+1$ runs and $Q$ is semistandard with weight $\lambda$. The inverse map will give a two line array that, when rearranged to give $1,\ldots, n$ on the bottom row, will give a descent in the top row between columns $i$ and $i+1$ exactly when $i+1$ starts a new run in $P$. The top row will also have weight $\lambda$, since $Q$ has weight $\lambda$. The decrementing process described above on $\lambda$-restricted permutations has a natural inverse, so we reverse that process and end up with a $\lambda$-restricted permutation with $k$ descents in the top row of the array as desired. Since RSK is a bijection, we know that both directions are injective, so the proof is complete.

\end{proof}

\begin{corollary}
It holds that
$$\sum_{|\lambda| = n} \sum_{k=0}^{n-1} A(\lambda, k){\alpha +n-1-k\choose n} m_\lambda = \sum_{|\nu|=n}\sum_{k=0}^{n-1} \textnormal{QYT}_{=k+1}(\nu){\alpha+n-1-k\choose n} s_\nu.$$
\end{corollary}
\begin{proof}
We proceed by induction on the poset of partitions induced by the dominance order. For a given partition $\lambda$, our inductive hypothesis is that the coefficient of $s_\nu$ matches the claim for all $\nu > \lambda$. From this, we show that the coefficient of $s_\lambda$ is correct as well. 

First we need the base case. $A((n),k)=1$ when $k=0$ and is zero otherwise. Therefore, the coefficient of $m_n$ on the left hand side is ${\alpha +n-1 \choose n}$. On the right hand side, we can only look to the expansion of $s_n$ to get an $m_n$ term, so it is clear that the coefficient of $s_n$ on this side must also be ${\alpha+n-1\choose n}$. By definition of quasi-Yamanouchi tableaux, QYT$_{=1}((n)) = 1$, which confirms the base case.

Now let $\lambda$ be an arbitrary partition of size $n$ and assume the inductive hypothesis. The expansion of Schur functions into monomial symmetric functions forces the coefficient of $m_\lambda$ on either side to be
$$\sum_{k=0}^{n-1} A(\lambda, k) {\alpha +n-1-k\choose n} = \sum_{k=0}^{n-1} C_{\lambda,k+1}{\alpha +n-1-k\choose n} + \sum_{\substack{|\nu|=n\\ \nu> \lambda}} K_{\nu\lambda} \sum_{k=0}^{n-1} \textnormal{QYT}_{=k+1}(\nu){\alpha+n-1-k\choose n},$$
where $\sum_{k=0}^{n-1} C_{\lambda,k+1}{\alpha +n-1-k\choose n}$ is the coefficient of $s_\lambda$ and the other sum comes from each $s_\nu$ for $\nu > \lambda$. Applying Lemma \ref{RSKlemma} immediately proves that $C_{\lambda,k+1} = \textnormal{QYT}_{=k+1}(\lambda)$, completing the inductive argument.
\end{proof}

Linking these together and applying Lemma \ref{QYTsymmetry} completes the proof of Theorem \ref{QYTtheorem}, which shows that the Schur expansion of $\tilde{J}_{(n)}^{(\alpha)}(X)$ is in fact a generating function for quasi-Yamanouchi tableaux up to a constant of $n!$, thus proving Conjecture 1 for the case of $\mu = (n)$. Chen, Yang, and Zhang \cite{CYZ16} adapted a result by Brenti \cite{Br89} to show that the polynomial
$$\sum_{T\in \textnormal{SYT}(\lambda)} t^{des(T)}$$
has only real zeroes. Using this and the definition of quasi-Yamanouchi tableaux, Theorem \ref{QYTtheorem} also proves the $\mu = (n)$ case of the second part of Conjecture 1.

\subsection{Fundamental quasisymmetric expansions}
One last application of Theorem \ref{QYTtheorem} takes us into a brief digression towards the fundamental quasisymmetric expansion. In the ${\alpha + k \choose n}$ basis, we can obtain the fundamental quasisymmetric expansion of $\tilde{J}_{(n)}^{(\alpha)}(X)$ as a corollary of the following result.

\begin{theorem}\label{QYTqsymexp} It holds that
$$\sum_{\pi \in S_n} t^{des(\pi)} F_{Des(P(\pi))}(x)= \sum_{|\mu|=n} \sum_{k=0}^{n-1}\textnormal{QYT}_{=k+1}(\mu)t^k s_\mu$$
where $P(\pi)$ is the insertion tableau of $\pi$ given by RSK.
\end{theorem}

\begin{proof}
Connect all $\pi \in S_n$ with colored edges corresponding to elementary dual equivalence involutions to get a graph $G$. By looking at properties of the bump paths in RSK, we can see that RSK respects dual equivalence relations in the $P$ insertion tableaux, so applying RSK to every vertex to get $G'$ maintains the edge relations between the $P$ tableaux.
For $\mu$ a partition, let $G_\mu$ be the dual equivalence graph on $\textnormal{SYT}(\mu)$. Each connected component of $G'$ will be isomorphic to $G_\mu$ for some $|\mu|=n$, and furthermore, there will be exactly $|\textnormal{SYT}(\mu)|$ copies of $G_\mu$ contained in $G'$ for each $\mu$.

Dual equivalence relations do not change the descent set of a permutation, and the number of descents of a permutation is equal to $\#\textnormal{runs}- 1$ of its $Q$ recording tableau. Therefore, since the descent set is constant on the vertices of a connected component of $G$, the number of runs of each corresponding $Q$ tableau is also constant. Among pairs $(P,Q)$ of shape $\mu$, $Q$ ranges over all $\textnormal{SYT}(\mu)$ with $|\textnormal{SYT}(\mu)|$ of each appearing, so a counting argument tells us there must be exactly $\textnormal{QYT}_{=k}(\mu)$ many connected components isomorphic to $G_\mu$ which have $k$ runs in each of the $Q$ tableaux of its vertices. Then taking the sum
$$\sum_{\pi \in S_n} t^{des(\pi)}F_{Des(P(\pi))}(x) = \sum_{(P,Q) \in G'} F_{Des(P)}(x)t^{des(Q)}$$
and applying the expansion of Schur functions into the fundamental quasisymmetric basis to the right hand side completes the proof.
\end{proof}

\begin{corollary} \label{jacksqsymmexp}
It holds that
$$\tilde{J}_{(n)}^{(\alpha)}(X) = n!\sum_{\pi \in S_n} {\alpha +n-1-des(\pi)\choose n} F_{Des(P(\pi))},$$
where $P(\pi)$ is the insertion tableau of $\pi$ given by RSK.
\end{corollary}
\begin{proof}
Apply Theorems \ref{QYTtheorem} and \ref{QYTqsymexp} and Lemma \ref{QYTsymmetry}.
\end{proof}

This result prompted the following conjecture on the quasisymmetric expansion for general partitions $\mu$. 

\begin{conjecture}
For a partition $\mu$ of size $n$, it holds that
$$\tilde{J}_\mu^{(\alpha)}(X) = \sum_{\pi,\tau \in S_n} {\alpha + n -1 - des(\pi) \choose n} F_{\sigma(\pi,\tau,\mu)}(x)$$
for some set-valued function $\sigma$ depending on $\pi, \tau,$ and $\mu$ and with image in $\{1,\ldots,n-1\}$.
\end{conjecture}

Corollary \ref{jacksqsymmexp} proves this conjecture in the case of $\mu=(n)$, where $\sigma(\pi,\tau,(n)) = Des(P(\pi))$, and Proposition \ref{eulerianstirlingprop} proves it in the case of $\mu = (1^n)$, where $\sigma(\pi,\tau,(1^n)) = \{1,\ldots,n-1\}$ for all $\pi,\tau \in S_n$. Furthermore, if we momentarily assume that the Jack polynomials are Schur positive in this basis, Corollary \ref{eulerianstirlingcor1} along with the expansion of Schurs into fundamental quasisymmetrics shows that this conjecture is true in general for some $\sigma$, although it does not tell us what $\sigma$ should be.
Finally, while the fundamental quasisymmetric expansion would be interesting in its own right, it may also lead to a proof of Schur positivity by a generalization of the method used in Theorem \ref{QYTqsymexp}. Corollary \ref{jacksqsymmexp} and Conjecture 5 have the following analogous conjectures in the ${\alpha \choose k} k!$ basis, where $B_n$ is the set of set partitions of $\{1,\ldots,n\}$.

\begin{conjecture}\label{conjbasis2}
For a partition $\mu$ of size $n$, it holds that
$$\tilde{J}_\mu^{(\alpha)}(X) = \sum_{\substack{\pi \in S_n \\ \beta \in B_n}} {\alpha \choose |\beta|}|\beta|! F_{\rho(\pi,\beta,\mu)}(x)$$
for some set-valued function $\rho$ depending on $\pi, \beta$, and $\mu$ with image in $\{1,\ldots,n-1\}$.
\end{conjecture}

For the $\mu=(n)$ case, we first define a function. Given $\pi \in S_n$ and $\beta \in B_n$, define $f_\beta(\pi)$ to be a rearrangement of $\pi$ so that if $\{b_1,\ldots,b_k\}\in \beta$, then $b_1,\ldots,b_k$ appear in increasing order in $f_\beta(\pi)$ without changing the position of the subsequence. For example, given $\beta = \{\{1,4\},\{2,3,5\}\}$ and $\pi = 24531$, $f_\beta(\pi) = 21354$.

\begin{conjecture}\label{conjbasis2muequalsn}
It holds that
$$\tilde{J}_{(n)}^{(\alpha)}(X) = \sum_{\substack{\pi \in S_n \\ \beta \in B_n}} {\alpha \choose |\beta|}|\beta|! F_{Des(P(f_\beta(\pi)))}(x),$$
where $P(f_\beta(\pi))$ is the insertion tableau of $f_\beta(\pi)$ given by RSK.
\end{conjecture}

Proposition \ref{eulerianstirlingprop} also proves the $\mu = (1^n)$ case here, and we can make similar remarks as above. That is, if we assume Schur positivity, that Conjecture \ref{conjbasis2} is true for some $\rho$ and that the fundamental quasisymmetric expansion could help prove Schur positivity in this basis.

\subsection{Rook Boards}
Returning to the problem of Schur positivity, we also had some success approaching the problem with rook boards. We first use Proposition \ref{rookboardprop} to obtain a combinatorial interpretation of the coefficient of $s_\lambda$ in our binomial bases when $\mu = \lambda$ for a hook shape. In general for $\mu = \lambda$, the coefficient of $s_\mu$ is the same as the coefficient of $m_\mu$ in the monomial expansion, so we can obtain from the combinatorial formula for the monomial expansion \cite{KnSa97} that
$$\langle \tilde{J}_\mu^{(\alpha)}(X), s_{\mu}\rangle = \prod_{s\in \mu} (arm(s) + \alpha (leg(s) +1)).$$
When $\mu = (n-\ell, 1^\ell)$ is a hook shape, this product becomes
\begin{align*}
\langle \tilde{J}_\mu^{(\alpha)}(X),&\ s_{\mu}(X)\rangle = \ell!\alpha^\ell ((\ell+1)\alpha + (n-1))(\alpha + (n-2))\cdots(\alpha + 1) \alpha\\
=& \ell\cdot \ell!(\alpha + (n-\ell-2))\cdots(\alpha + 1)\alpha^{\ell+2}+ \ell!(\alpha+(n-\ell-1))\cdots(\alpha + 1)\alpha^{\ell+1},
\end{align*}
then applying Proposition \ref{rookboardprop} gives the following result.

\begin{proposition} For $\mu = \lambda = (n-\ell, 1^\ell)$, we have
\begin{align*}
\langle \tilde{J}_\mu^{(\alpha)}(X),s_\mu \rangle &= \sum_{k=0}^n (\ell \cdot \ell! h_k(B(c_1,\ldots,c_n)) + \ell! h_k(B(d_1,\ldots,d_n))) {\alpha + k \choose n }\\
& = \sum_{k=0}^n (\ell \cdot \ell! r_k(B(c_1,\ldots,c_n)) + \ell! r_k(B(d_1,\ldots,d_n))) {\alpha \choose k }k!\\
\end{align*}
where $c_1 = c_2 = \cdots = c_{n-\ell-1} = n-\ell-2$ and $c_{n-\ell+i} = n-\ell-1+i$ for $0 \geq i \geq \ell$ and $d_1 = d_2 = \cdots = d_{n-\ell} = n-\ell-1$ and $d_{n-\ell+i} = n-\ell-1+i$ for $1 \geq i \geq \ell$.
\end{proposition}

\begin{example}$B(c_1,\ldots,c_4)$ and $B(d_1,\ldots,d_4)$ for $\mu=\lambda=(3,1)$.
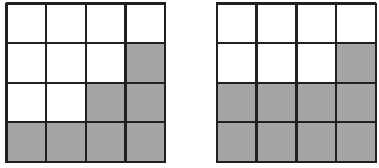
\begin{displaymath}
\tableau{
\ & \ &\ &\ \\
\ & \ & \ & \graybox\\
\ & \ & \graybox & \graybox \\
\graybox & \graybox & \graybox & \graybox \\}
\ \ \ \ \ 
\tableau{
\ & \ &\ &\ \\
\ & \ & \ & \graybox\\
\graybox & \graybox & \graybox & \graybox \\
\graybox & \graybox & \graybox & \graybox \\}
\end{displaymath}
\end{example}

This approach also yields a combinatorial interpretation for both bases in the case of $\mu = (n)$. We can obtain $J_{\mu}^{(1/\alpha)}(X)$ via the specialization\\

\noindent\resizebox{1.0\textwidth}{!}{
$J_{\mu}^{(1/\alpha)}(X) = \lim_{t\rightarrow 1} \frac{J_{\mu}(X;t^{1/\alpha},t)}{(1-t)^n} = \lim_{q\rightarrow 1} \frac{J_{\mu}(X;q,q^\alpha)}{(1-q^\alpha)^n}=\lim_{q\to 1} \frac{J_{\mu}(X;q,q^\alpha)}{(1-q)^n}\frac{(1-q)^n}{(1-q^\alpha)^n}=\frac{1}{\alpha^n}\lim_{q \to 1} \frac{J_{\mu}(X;q,q^\alpha)}{(1-q)^n}$}
\\

\noindent so that 
$$\tilde{J}_\mu^{(\alpha)}(X) = \lim_{q\to 1} \frac{J_{\mu}(X;q,q^\alpha)}{(1-q)^n}.$$
Then when $\mu = (n)$, we can apply the limit as $q \to 1$ 
to a result of Yoo \cite[Theorem 3.2]{Yoo12} to obtain
$$\lim_{q\to 1}\frac{J_{(n)}(X;q,q^\alpha)}{(1-q)^n} = \sum_{|\lambda|=n}s_{\lambda} K_{\lambda,1^n} \prod_{(i,j)\in \lambda}(\alpha+i-j),$$
where $(i,j)\in \lambda$ refers to cells of the diagram of $\lambda$ identified with their Cartesian coordinates. Arrange the values of $i-j$ in non-increasing order and rewrite to get the desired $\prod_{i=1}^n (\alpha + c_i-i+1)$ form. It is clear that for any partition $\lambda$, this produces a sequence $0\leq c_1 \leq \cdots \leq c_n \leq n$, so we can apply Proposition \ref{rookboardprop} again to obtain the following.

\begin{theorem}\label{rookboardtheorem}
It holds that 
$$\langle \tilde{J}_{(n)}^{(\alpha)}(X),s_\lambda\rangle = \sum_{k=0}^n K_{\lambda,1^n} h_k(B(c_1,\ldots,c_n)) {\alpha + k \choose n} = \sum_{k=0}^n K_{\lambda,1^n} r_{n-k}(B(c_1,\ldots,c_n)) {\alpha \choose k}k!,$$
with $c_1,\ldots,c_n$ given above.
\end{theorem}
\begin{example} The values $i-j$ and $B(c_1,\ldots,c_5)$ for $\lambda = (3,2)$.
 \begin{displaymath}
    \raisebox{-45pt}{\tableau{
    -1 & 0 \\
    0 & 1 & 2 \\
    }}\ \ \ \ \ 
    \tableau{
    \ & \ & \ & \ & \ \\
    \ & \ & \ & \ & \ \\
    \ & \ & \ & \graybox & \graybox \\
    \graybox & \graybox & \graybox & \graybox & \graybox \\
    \graybox & \graybox & \graybox & \graybox & \graybox \\}
    \end{displaymath}
\end{example}

In \cite{HOW} it is shown that the rook and hit polynomials of Ferrers boards have only real zeros, so the two results of this section also prove the second part of Conjecture \ref{C1}  for these special cases. 
We note that Theorem \ref{rookboardtheorem} provides a very different looking combinatorial interpretation to the one seen in Theorem \ref{QYTtheorem} for the ${\alpha + k \choose n}$ basis. It would be interesting to find a bijection between the rook board interpretation of Theorem \ref{rookboardtheorem} and the tableau interpretation of Theorem \ref{QYTtheorem}. This may lead to a tableau interpretation for the ${\alpha\choose k } k!$ basis, but also would be of independent interest in understanding the relationship between tableaux and rook boards.

\section{Acknowledgements}
We would like to thank Greta Panova, Emily Sergel, and Andy Wilson for their many helpful discussions and suggestions. Per Alexandersson was partially supported by the Knut and Alice Wallenberg Foundation, James Haglund was partially supported by NSF grant DMS-1600670, and George Wang was partially supported by the NSF Graduate Research Fellowship, DGE-1321851. This paper is the full version of an extended abstract, On the Schur expansion of Jack Polynomials \cite{FPSAC}, that is to appear in the 2018 FPSAC proceedings.

\printbibliography

\end{document}